\documentclass[]{interact}

\usepackage{epstopdf}
\usepackage{subfigure}

\usepackage{shorthand}

\usepackage[numbers,sort&compress]{natbib}
\bibpunct[, ]{[}{]}{,}{n}{,}{,}
\renewcommand\bibfont{\fontsize{10}{12}\selectfont}
\makeatletter
\def\NAT@def@citea{\def\@citea{\NAT@separator}}
\makeatother

\theoremstyle{plain}
\newtheorem{thm}[theorem]{Theorem}

\newtheorem{propo}[theorem]{Proposition}
\newtheorem{rema}[theorem]{Remark}
\newtheorem{cor}[theorem]{Corollary}

\theoremstyle{definition}
\newtheorem{definition}[theorem]{Definition}
\newtheorem{example}[theorem]{Example}

\theoremstyle{remark}

\begin{document}


\title{Gramian Tensor Decomposition via Semidefinite Programming}

\author{
\name{Erik Skau\textsuperscript{a}\thanks{This research was partly supported by NSF grant CCF-1217557. CONTACT: Erik Skau. Email: ewskau@gmail.com} and Agnes Szanto\textsuperscript{a}}
\affil{\textsuperscript{a}North Carolina State University, Raleigh, NC}
}

\maketitle

\begin{abstract}
In this paper we examine a symmetric tensor decomposition problem, the Gramian decomposition, posed as a rank minimization problem. We study the relaxation of the problem and consider cases when the relaxed solution is a solution to the original problem. In some instances of tensor rank and order, we prove generically that the solution to the relaxation will be optimal in the original. In other cases we present interesting examples and approaches that demonstrate the intricacy of this problem.
\end{abstract}


\section{Introduction}

\subsection{Background}
Let $\A\in \F^{(n+1)\times \cdots\times (n+1)}$ be a $D$-way, or order $D$, symmetric tensor over a field $\F$ of size $(n+1)\times\cdots\times (n+1)$ ($D$-times). Let  $R:= \F[x_1, \ldots, x_n]$ and let $R_D$ denote the set of polynomials of  degree at most $D$ in $R$. Then we can associate to $\A$ a polynomial 
\begin{eqnarray}\label{qbeta}
p = \sum_{\beta\in \N^{n}, |\beta|\leq D} {{D}\choose{D-|\beta|, \beta_1,\ldots, \beta_n}} p_\beta x^\beta\in R_D
\end{eqnarray}
 by simply multiplying  $\A$ by the vector $[1, x_1, \ldots, x_n]$ from all the $D$ directions. This gives a bijection between symmetric $D$-way tensors over $\F$ and polynomials in $R_D$. 

We define the  {\em symmetric rank} of the tensor $\A$, and the {\em rank} of the polynomial $p$ as follows:

\begin{definition} We say that $\A\in \F^{(n+1)\times\cdots\times (n+1)}$ has {\em symmetric rank} $r$   if there exist  distinct $\bvec{v}_1=(v_{1,0}, v_{1,1}, \ldots, v_{1,n}),$ $ \ldots, $$\bvec{v}_r=(v_{r,0}, v_{r,1}, \ldots, v_{r,n}) \in \overline{\F}^{n+1}$ with coordinates from the algebraic closure $\overline{\F}$ of $\F$, and $\lambda_1, \ldots, \lambda_r\in \overline{\F}/\{0\}$  such that $r$ is minimum and
\begin{eqnarray}\label{Adecomp}
\A=\sum_{t=1}^r\lambda_t\bvec{v_t}^{\otimes D} :=\sum_{t=1}^r\lambda_t\left[  v_{t, i_1}\cdots v_{t, i_D}\right]_{i_1, \ldots, i_D=0}^n. 
 \end{eqnarray}
Equivalently, we say that  $p\in R_D$ has {\em rank} $r$ if 
$r$ is minimal and \begin{eqnarray}\label{decomp}
p = \sum_{t=1}^r \lambda_t L_{\bvec{v}_t}^D,
\end{eqnarray}
where $ L_{\bvec{v}_t} (x_1, \ldots, x_n):=v_{t,0}+ v_{t,1}x_1+\cdots+v_{t,n}x_n$ is the linear form associated to $\bvec{v}_t=(v_{t,0}, v_{t,1},\ldots,v_{t,n})$ for $t=1, \ldots, r$. The expressions in (\ref{Adecomp}) or (\ref{decomp}) are called the {\em rank $r$ symmetric decompositions} of $\A$ and  $p$, respectively.
\end{definition}

There are different, non-equivalent notions of tensor rank in the literature, such as the multilinear rank or non-symmetric rank, etc. (see \cite{ComGolLimMour2008}). Also, one can define the symmetric rank over non-algebraically closed fields, which unlike for matrices, may differ from the above defined symmetric rank  for tensors of order $>2$.  If the field $\F$ is the set of real numbers and the order $D=2d$ is even, we can define the {\em Gramian rank} as follows:

\begin{definition} Let $\A\in \R^{(n+1)\times\cdots\times (n+1)}$ be a real symmetric tensor of order $2d$ and $p\in R_{2d}$  be the corresponding real polynomial.  We say that $\A$ and $p$ is {\em Gramian} with {\em Gramian rank} $r$ 
if there exist  distinct $\bvec{v}_1=(v_{1,0}, v_{1,1}, \ldots, v_{1,n}),$ $ \ldots, $$\bvec{v}_r=(v_{r,0}, v_{r,1}, \ldots, v_{r,n}) \in \R^{n+1}$ and $\lambda_1, \ldots, \lambda_r\in \R_{>0}$ positive real numbers  such that $r$ is minimal and (\ref{Adecomp}) or (\ref{decomp}) holds. The decompositions in (\ref{Adecomp}) and (\ref{decomp}) are called the {\em Gramian decompositions} of $\A$ and  $p$, respectively.
\end{definition}

In this paper we consider the problem of finding the Gramian rank and decomposition for a real symmetric tensor of order $2d$, or equivalently, for a polynomial of degree $2d$. Note that not all polynomials of degree $2d$ are Gramian, in particular, Gramian polynomials are a subset of sum of square polynomials. Hillar and Lim in \cite{HillarLim2013} proved that deciding whether a tensor/polynomial is Gramian is NP-hard even for $d=2$. Also note that even if a tensor is Gramian, its Gramian rank may be much higher than its symmetric rank. 

We give an algorithm that finds the Gramian decomposition in the case when the Gramian rank is sufficiently small. Our approach is to use a relaxation of this problem to semidefinite programming and to show that for sufficiently small Gramian rank $r$ the optimum of the relaxed problem gives a  Gramian decomposition of  length  $r$. This work is a first step to attack the more general problem of finding the symmetric rank and decomposition via semidefinite relaxation. The general case is subject to future research. 

The main results of this paper are as follows:
\begin{itemize}
\item We give a meaningful semidefinite relaxation of the problem of finding the Gramian rank and decomposition of a polynomial $p\in R_{2d}$, assuming that its Gramian rank is sufficiently small. The relaxation becomes a matrix completion problem of moment matrices with minimal trace. 
\item We simplify  and interpret the condition that a given moment matrix is the optimum of our relaxed semidefinite program, using special properties of the dual of the semidefinite program.
\item We analyze special cases when we can guarantee that a given moment matrix is the optimum of the relaxed semidefinite program. In these special cases we point to a connection to the theory of the regularity index of overdetermined polynomial systems. Using this theory we list triples $(n,d,r)$ where we can prove that the optimum of the semidefinite relaxation corresponds to the Gramian decomposition of rank $r$ of a polynomial of degree $2d$ in $n$ variables. 
\end{itemize}

\subsection{Related Work}  Motivation for looking at the  tensor decomposition problem comes from its broad  application areas. The earliest results on tensor decomposition were applications in mathematical physics (\cite{Hitch1927a,Hitch1927b});  psychometrics (\cite{Tucker1966,Tucker1963,CarChan1970,Harshman1970,CaPrKru1980}); algebraic complexity theory (\cite{Knuth69,Strassen69,Krus1977,How1978,Land2006,Land2008});  and in chemometrics (\cite{AppDav1981,Gel1989,SmiBroGel2004}). In higher order statistics, moments and cumulants  are intrinsically tensors (cf.  \cite{McCullagh87}). Symmetric tensor decomposition is proven to be useful in {\em blind source separation} techniques, which are capable
of identifying a linear statistical model only from its outputs (cf. \cite{Com1992}). These blind identification techniques in turn are very popular in numerous applications, including  telecommunication (\cite{SiBroGi2000,GrCoMoTre2002,AlFeCoCh2004}); radar (\cite{ChaCoMu1993});  biomedical engineering (\cite{DeLDeMVan2000});    image and signal  processing (\cite{DeLa1997,GiaHea2000}) just to name a few. An excellent survey of more recent applications of tensor methods can be found in  \cite{HillarLim2013}.  \\

Despite the rich literature on the numerical aspects of the  symmetric tensor decomposition problem, there are relatively small numbers of publications concerned with the symbolic computational aspects of computing the rank of symmetric and non-symmetric tensors. Even though the first algorithm solving the problem  in the bivariate symmetric case goes back to Sylvester \cite{Syl1886},  and several other symbolic algorithms exist in the literature for finding the rank of symmetric tensors (see for example \cite{Reich1992,ComMou96,IaKa1999,BrCoMoTs2009, LanTeit2010,LansOct2010a,LansOct2010b,OedOtta2011}) and non-symmetric tensors (see for example \cite{Strassen83,LathMoorVan2004,AlFeCoCh2004,Lath2006,LanMan2008,OedOtta2011,BerBraComMou2011,BerBarComMou2011b}), they all have strong constrains on the degree $d$, dimension $n$ and/or on the rank $r$. A list of all  cases where we know the  defining equations for (border)-rank $r$ symmetric tensors can be found in \cite{LansOct2010a}. Symmetric rank computation is NP-hard \cite{HillarLim2013}, and its approximation doesn't always exist as the set of rank $r$ tensors is not closed \cite{SilvaLim2008} \\
 
As we will see in the preliminaries below there is a close relationship between the so called truncated moment problem and the Gramian decomposition of tensors.  Here we only mention work that is closest to our problem, namely when representing measures that are finitely atomic. The foundations of the theory and algorithms to study this truncated moment problem were laid down in a sequence of work by Curto and Fialkow in \cite{CurtoFialkow1996,CurtoFialkow1998}, including the so called stopping criteria that we use in this paper. In a series of papers \cite{Lasserre2001,Lasserre2010,LaLaRo2008,LaLaRo2009} the moment problem is connected to polynomial optimization and the solution of polynomial systems over the reals, and our approach is based on this work. The direct relationship between symmetric tensor decomposition and the truncated moment problem was described in the works \cite{BerBarComMou2011b,BrCoMoTs2009}; our approach strongly relies on these results. As we mentioned earlier, in \cite{HillarLim2013} they prove that detecting if a symmetric tensor is Gramian is NP-hard, and they also discussed the relationship between Gramian, non-negative definite tensors, and completely positive matrices. Reznick in \cite{Reznick92} proved that the cone of tensors and of Gramian tensors are dual. It is also proved here that the set of Gramian rank $r$ tensors is closed. In \cite{LimComon2010} they  deduce a computationally feasible condition for  uniqueness using  the notion of coherence. In \cite{LimComon2009} they study nonnegative approximations of nonnegative tensors, where they use a  generalization of the notion of completely positive matrices, which is different from Gramian and nonnegative-definite tensors.   \\

Relaxations of  matrix rank minimization problems using the nuclear norm of matrices was first introduced  in \cite{FaHiBo2001,Fazel2002}.  There is a rich literature on results about the accuracy of the relaxation of a low rank optimization problem using the nuclear norm.  The low rank matrix completion approach assumes that a linear  image of the underlying low rank matrix $M$ is known and attempts to recover the full matrix $M$.  The motivation and  justification for this relaxation is that the nuclear norm of matrices is the convex envelope of the rank function  (cf. \cite{RechFazParr2010}).   The main  results in \cite{CandesRecht2009,CanTao2010,RechFazParr2010}  give  general assumptions which guarantee both the rank minimization problem and its relaxation to have $M$ as its unique solution (with high probability). One of these assumptions in \cite{CandesRecht2009} is the existence of a bound on the so called {\em coherence} of the column and row spaces of the output $M$.  Another such assumption is given for the input. In \cite{RechFazParr2010} they show that if a certain {\em restricted isometry property} holds for the linear transformation defining the constraints, the minimum-rank solution can be recovered by the nuclear norm relaxation. Similar ideas were explored in \cite{NavDeL2009,LiuMusWonYe2009,GaReYa2011,SigDeLSuy2011,GoldfarbQin2014} to recover low {\em multilinear rank} tensors. Here the objective function is the sum of the ranks of the flattenings of the tensor which is subject to linear constrains. This is relaxed by using the sum of the  nuclear norms of the flattenings instead. 
Our approach is closest to the work in   \cite{SCPW2012}, where they study conditions when the semidefinite relaxation solves the minimal rank matrix diagonal completion problem. They reinterpret the dual of the semidefinite relaxation problem in several different ways and connect their original problem to other well-studied problems in statistics and geometry. We follow a similar approach, but leading to very different results.  
\\

\section{Preliminaries}

Before describing  our results, let us give a  brief summary of the main results in the theory of flat extensions of moment matrices (see \cite{BrCoMoTs2009} for more details).  Assume that we have a Gramian decomposition as in (\ref{Adecomp}) or (\ref{decomp})  for some $\bvec{v}_1=(v_{1,0}, v_{1,1}, \ldots, v_{1,n}),$ $ \ldots, $$\bvec{v}_r=(v_{r,0}, v_{r,1}, \ldots, v_{r,n}) \in \R^{n+1}$ and $\lambda_1, \ldots, \lambda_r\in \R_{>0}$. We assume that
$$
v_{1,0}=1, \ldots, v_{r,0}=1
$$
and denote by 
$$\bvec{z}_i=( v_{i,1}, \ldots, v_{i,n})\in \R^{n} \text{ for }i=1, \ldots, r.
$$

Consider the infinite matrix $M$ and its truncation ${M}_{i, j}$ for some $i, j\in \N$ defined by 
 \begin{eqnarray}\label{moment}
M:=\left[ m_{\beta+\beta'}\right]_{\beta, \beta' \in \N^{n}}  \text{ and } {M}_{i, j}:=\left[ m_{\beta+\beta'}\right]_{|\beta|\leq i,|\beta'|\leq j},
\end{eqnarray}
where for $\alpha\in \N^{n}$
$$
m_{\alpha}=\sum_{t=1}^r \lambda_t \bvec{z}_i^{\alpha}, 
$$
denotes the {\em moments} corresponding to the points $\{\bvec{z}_1,\ldots, \bvec{z}_r\}$. If $i=j$ we will denote ${M}_{i, i}$ by simply $M_i$. These matrices have  so called quasi-Hankel structure (see \cite{mourrain1996multidimensional}), and  called {\em moment matrices}, i.e.  they are   matrices whose rows and columns are indexed by monomials and the  entries depend only  on the product of the indexing monomials. 

Let   $V:= [ \bvec{z}_i^\beta]_{i=1, ..r, \beta\in \N^n}$ be the Vandermonde matrix with infinitely many columns, its truncation $V_i:= [ \bvec{z}_i^\beta]_{i=1, ..r, \beta\in \N^n, |\beta|\leq i}$, and let $\Lambda:={\rm diag} (\lambda_1, \ldots, \lambda_r)$. Then we have 
\begin{eqnarray}\label{Vandermonde}
M=V^T\Lambda V\quad {M}_{i, j}= V_i^T\Lambda V_{j}.
\end{eqnarray}

When we only know the tensor $\A$ or the polynomial $p$ as in (\ref{qbeta}) for $D=2d$, but not the decomposition,  
from (\ref{decomp}) it is easy to see that for $|\beta+\beta'|\leq 2d$ we have the following relationship between the entries of the moment matrix $M$ and the coefficients of $p$:
\begin{eqnarray}\label{m-p}
[M]_{\beta, \beta'}=m_{\beta+\beta'}={{2d}\choose {\beta+\beta'}}^{-1} p_{\beta+\beta'}\quad |\beta+\beta'|\leq 2d.
\end{eqnarray} 
The truncations $M_{i,j}$ for $i+j\leq 2d$, are called {\em catalecticant matrices}, and its theory  goes back to  Sylvester in \cite{Syl1886}. 
Note that for $i=j=d$, $M_d$ is a  symmetric matrix of size
$
 {{n+d}\choose d} = \dim R_d.
$

Next we define the notion of  flat extensions of moment matrices:

\begin{definition}
Given $M_D$ a moment matrix for some degree $D\geq 0$ as in (\ref{moment}). We call an infinite moment matrix $M$  an {\em   extension} of $M_D$  if $$[M]_{\beta,\beta'} = [M_D]_{\beta,\beta'} \text { for } |\beta+\beta'|\leq D. $$
If, in addition,
$$
 \rank{M}=\rank{M_D}, 
$$ then we say that $M$ is a {\em flat extension} of $M_D$.   Furthermore, if $M$ is positive semidefinite, we call $M$ a {\em  Gramian  flat extension} of $M_D$. 
\end{definition}

Clearly, if  $p\in R_{2d} $ has symmetric rank $r$, then there exists at least one infinite moment matrix $M$ of rank $r$ that extends $M_d$. Similarly, if $p$ has Gramian rank $r$ then there exists some positive semidefinite moment matrix  $M$  of rank $r$ that extends $M_d$. If, in addition, $M_d$ also has rank $r$, then $M$ is a Gramian flat extension of $M_d$.     Note that if the  decomposition of $p$ is not unique, then the flat extensions of $M_d$ may not be unique either. 
The converse is not entirely true: if $M_d$ has an infinite  flat extension  $M$ of  rank $r$, then $p$ has   a so called {\em generalized decomposition}, where the points $\{\bvec{z}_t\}_{t=1}^r$ may be repeated (see \cite{BerBraMou2014} for more details).  However, for a positive semidefinite flat extension, the corresponding points in the decomposition are always distinct. Thus, these positive semidefinite flat extensions always correspond to a Gramian decomposition of the tensor \cite{CurtoFialkow1996}.

In \cite{CurFial2005,BrCoMoTs2009,BerBarComMou2011b}  they  give conditions for the existence of a (Gramian) flat extension  in terms of finite truncations of $M$:\\
 
\begin{thm}[{\sc Stopping criterion for flat extension}]\label{thm:stop}  Let  $M_d$ be a moment matrix as above. Let  $M$ be an infinite extension of $M_d$ as above.   $M$ has rank $r$ if and only if  there exist  $D\geq 0$ such that $$
\mathrm{rank} (M_D)=\mathrm{rank} (M_{D+1})=r. 
$$
 If, in addition,  $M_{D+1}$ is positive semidefinite, then  $M$ is also positive semidefinite. We call $M_{D+1}$ a {\em truncated (Gramian) flat extension} of $M_D$. 
 \end{thm}
 
 Note that once the above stopping criterion is satisfied, one can compute a system of multiplication matrices  from the kernel of $M_{D+1}$, and the coordinates of the points $\bvec{z}_i$ for $i=1, \ldots, r $ can be read out from the eigenvalues of these  multiplication matrices \cite{CurFial1996,LaLaRo2008,LaLaRo2009}.
 
 In the present paper we assume that $p\in R_{2d}$ has Gramian rank $r$ satisfying
 $$
{\rm size}(M_{d-1})< r\leq  {\rm size}(M_{d})
$$ and  $M_d$ has a truncated Gramian flat extension $M_{d+1}$  of rank
 $$
\mathrm{rank} (M_{d})=\mathrm{rank} (M_{d+1})=r,
$$
i.e. $D=d$ in the stopping criterion above. 

However, given $p\in R_{2d}$, we only know the entries of  $M_d$, so  we want to find a  truncated Gramian flat extension $M_{d+1}$. Note that if $r\leq {\rm size}(M_{d-1})$ then by the stopping criterion we do not need to extend the matrix $M_d$ to find the Gramian rank.
So the {\em truncated Gramian  flat extension problem} that we attempt to solve in this paper is the following:\\

\begin{definition}[{\sc  Truncated Gramian flat extension problem}]\label{TNFEP} Given $p\in R_{2d}$ as in (\ref{qbeta}) with non-zero constant term. Assume that the corresponding truncated moment matrix $M_d$ given by (\ref{m-p}) has rank $r$ and is positive semidefinite. Find a positive semidefinite moment matrix extension $M_{d+1}$ of $M_d$ which has rank $r$, if one exists.  Equivalently, find a minimal rank positive semidefinite extension $M_{d+1}$ of $M_d$.
 \end{definition}

Unfortunately, the minimal rank optimization problem is NP-hard, and all known algorithms which provide exact solutions are double exponential in the dimension of the matrix (cf. \cite{CandesRecht2009}).  However, relaxation techniques were successfully applied for  ``low rank matrix completion''  or  ``affine rank minimization'' problems that are very similar in structure to our problem. Namely, the constraints on the extension matrix $M_{d+1}$ are all linear equalities. These relaxation techniques  replace  rank minimization by the minimization of the nuclear norm of the matrix. Recall that the nuclear norm of a matrix $M$ is defined by
\vspace{-3mm} $$
\|M\|_*:= \sum_{i=1}^r \sigma_i,
\vspace{-3mm} $$
where $\sigma_1>\sigma_2>\cdots >\sigma_r>0$ are the non-zero singular values of $M$. The advantage is that the nuclear norm is a convex function and can be optimized efficiently using semidefinite programming. Note that when $M$ is positive semidefinite then 
$$
\|M\|_*= {\rm trace} (M).
$$

\begin{definition}[{\sc  Relaxation of truncated Gramian flat extension}] \label{RTNFE} Given  $p\in R_{2d}$ with non-zero constant term, find a positive semidefinite moment matrix $M_{d+1}$ satisfying   $[M_{d+1}]_{\beta,\beta'}={{2d}\choose {\beta+\beta'}}^{-1} p_{\beta+\beta'}$ for $|{\beta+\beta'}|\leq 2d$,  and ${\rm trace} (M_{d+1})$ is minimal.
\end{definition}

The purpose of this paper is to prove that for sufficiently low Gramian rank $r$, the optimum of the relaxation is the minimal rank solution.

In \cite{CandesRecht2009,CanTao2010,RechFazParr2010} the goal of the low rank matrix completion and affine rank minimization problems is to give conditions on the matrix and on the linear constraints  so that the optimum of the minimal rank problem is unique and equal to the optimum of the nuclear norm relaxation. In our case uniqueness cannot always be expected, since symmetric tensors can have many minimal decompositions, resulting in different flat extensions of the same rank.   For example, if $r$ is the generic rank as in \cite{AlHir1995},   \cite{Mella2009} conjectures that the solution is never unique, except for  three cases.  The lack of uniqueness is a significant obstacle for the relaxation to find the minimal rank solution as the set of minimal rank decompositions may be a non-convex object. For this reason we cannot expect to find the minimal rank decomposition via semidefinite optimization. To address this obstacle we constrain ourselves to cases where the minimal decomposition of the symmetric tensor is essentially unique (up to unimodulus scaling).

For symmetric decompositions rather strong uniqueness results were proved in \cite{IaKa1999,ChiCil2006,Mella2009}. Namely, for a decomposition as in (\ref{decomp}),  if $d \geq 2$, and 
\begin{eqnarray}
r\leq {{d+n}\choose{d}}-n+1 = \dim R_{d} -n+1
\vspace{-3mm}\end{eqnarray}
then the decomposition is essentially unique, as long as the points $\{\bvec{ z}_i\}_{i=1}^r$ are in general position (cf.  \citep[Th.2.6][]{IaKa1999}).  Our ultimate goal would be to prove that in the cases of  unique decomposition, the semidefinite relaxation gives the minimal rank solution.  At this point we could only prove a small portion of these cases, however, in the process we uncovered some interesting connections of this problem to the theory of the regularity index of polynomial systems, which is an active research area in mathematics.

A difference between our problem and the ones considered in \cite{CandesRecht2009,CanTao2010,RechFazParr2010} is that the linear constraints on the extension $M_{d+1}$ are not given at random, and we cannot expect that the corresponding linear map  would satisfy  either the restricted isometry conditions of \cite{RechFazParr2010} or the injectivity when restricted to  the tangent space of rank $r$ matrices at the optimum as in \cite{CandesRecht2009}.  Thus to tackle our problem we needed new ideas. As we mentioned in the Introduction,  our approach is closest to the one in \cite{SCPW2012}, where they give equivalent interpretations for the dual of the relaxed semidefinite program, discovering interesting connections of  the original problem to other problems in geometry and statistics that were previously studied.

\section{Relaxation and Dual Problem}
Given $d \in \N$,   $n\geq 1$,  and $p\in R_{2d}$ as in (\ref{qbeta}), and let
$$
M_d = \left[m_{\beta+\beta'}\right]_{\substack{\beta, \beta'\in \N^n\\ |\beta+\beta'|\leq 2d}}
$$ 
be the corresponding truncated moment matrix as in (\ref{m-p}) with moments $m_\alpha={{2d}\choose {\alpha}}^{-1} p_{\alpha}$ for $|\alpha|\leq 2d$. Denote by  $$N:=  \dbinom{n+d+1}{n}=\dim R_{d+1},
$$ and by ${\mathcal S}_N$ the space of real symmetric matrices of size $N$. The truncated Gramian flat extension problem in Definition \ref {TNFEP} is finding a symmetric  matrix $X\in {\mathcal S}_N$, with columns and rows indexed by $\alpha, \beta\in \N^n$, such that
\begin{align*}
\min & \hspace{5mm} \rank{X}\\
\st & \hspace{5mm}
\begin{cases}
[X]_{ \beta, \beta'} = m_{\beta+\beta'}\ &\text{for} \ \vert \beta+\beta' \vert \leq 2d\\
[X]_{\beta ,\beta'} - [X]_{\gamma,\gamma'} = 0 \ &\text{if} \ \beta+\beta' = \gamma + \gamma'\\
X \succeq 0
\end{cases}\\
\end{align*}

Using the bilinear form 
$$
<A,B>:=Tr(A \cdot B), 
$$ we choose an orthonormal basis for the space of  symmetric matrices ${\mathcal S}_N$ as specified in Definition \ref{orth}. 

\begin{definition}[{\sc Choice of Orthogonal basis for ${\mathcal S}_N$}] \label{orth} For each $\alpha\in \N^n$ such that $|\alpha|\leq 2d+2$, we define the subspace ${\mathcal S}_\alpha\subset {\mathcal S}_N$ of symmetric matrices with support indexed by the set of pairs $\{(\gamma,\delta)\in \left(\N^{n}\right)^2 \;:\; \gamma + \delta = \alpha\}$. Fix $Y_\alpha \in {\mathcal S}_\alpha$ to be the {\em moment matrix} which has $1$ at each entry in its support. Then choose an arbitrary orthonormal basis  $\{ Z_{\alpha, i}\;:\; 1\leq i\leq\dim {\mathcal S}_\alpha-1\} \subset {\mathcal S}_\alpha $ for the subspace of ${\mathcal S}_\alpha$ orthogonal to $Y_\alpha$. 
\end{definition}
\begin{example}
For example, in the univariate case with a monomial basis of $1, x, x^2$ we define an orthogonal decomposition of $\mathcal{S}_3$:\\
$Y_0 = \begin{bmatrix}
1 &&\\
&&\\
&&
\end{bmatrix}, Y_1 =\begin{bmatrix}
&1&\\
1&&\\
&&
\end{bmatrix}, Y_2 =\begin{bmatrix}
&&1\\
&1&\\
1&&
\end{bmatrix}, Y_3 =\begin{bmatrix}
&&\\
&&1\\
&1&
\end{bmatrix}, Y_4 =\begin{bmatrix}
&&\\
&&\\
&&1
\end{bmatrix}$\\
One can easily see that this set is a basis for all 3 by 3 Hankel matrices. We also see that with our monomial list there are two ways to obtain $x^2 = x^2 \cdot 1 = x \cdot x$ so we then define one matrix orthogonal to $Y_2$ with respect to our inner product and with the same support $Z_2 =\begin{bmatrix}
&&-1\\
&2&\\
-1&&
\end{bmatrix}.$
\end{example}
Using this notation we rewrite the truncated Gramian flat extension problem as follows:

\begin{align*}
\min & \hspace{5mm} \rank{X}\\
\st & \hspace{5mm}
\begin{cases}
<Y_{\alpha},X>=m_{\alpha}\quad |\alpha|\leq 2d\\
<Z_{\alpha,i},X>=0\quad |\alpha|\leq 2d+2, \quad 1 \leq i \leq \operatorname{dim}({\mathcal S}_\alpha)-1\\
X \succeq 0
\end{cases},
\end{align*}

This we relax to a semidefinite program:
\begin{align*}
\min & \hspace{5mm} <I,X>\\
\st & \hspace{5mm}
\begin{cases}
<Y_{\alpha},X>=m_{\alpha}\quad |\alpha|\leq 2d\\
<Z_{\alpha,i},X>=0\quad |\alpha|\leq 2d+2, \quad 1 \leq i \leq \operatorname{dim}({\mathcal S}_\alpha)-1\\
X \succeq 0
\end{cases},
\end{align*}

Thus we get the following primal and dual semidefinite optimization problems (in standard form):

$$\begin{array}{c|c}
Primal & Dual \\
\hline
\min_{X} <I,X> & \max_{(\bvec{y},\bvec{z},S)} \sum m_\alpha  y_\alpha\\ 
\st \begin{cases}
<Y_{\alpha}, X> = m_{\alpha}\\
<Z_{\alpha,i}, X> = 0\\
X \succeq 0
\end{cases} &
\st \begin{cases}
S= I-\sum y_{\alpha} Y_{\alpha} -\sum z_{\alpha,i} Z_{\alpha,i} \\
S \succeq 0
\end{cases}
\end{array},$$
where the indices of $y_\alpha$ and $Y_\alpha$ run through $ |\alpha|\leq 2d$, while the indices of $z_{\alpha, i}$ and  $Z_{\alpha,i}$ run through $|\alpha|\leq 2d+2$ and  $1 \leq i \leq \operatorname{dim}({\mathcal S}_\alpha)-1$, using the notation of Definition \ref{orth}.\\

 In the rest of this paper we will use the following notation for the above semidefinite programs:
\begin{align*}
(\mathcal{P}) \; &: \; \text{primal problem in standard form;}\\
(\mathcal{D}) \; &: \; \text{dual problem in standard form;}\\
\mathcal{P} \; &: \; \text{feasible set of problem ($\mathcal{P}$);}\\
\mathcal{D} \; &: \; \text{feasible set of problem ($\mathcal{D}$);}\\
\mathcal{P}^* \; &: \; \text{optimal set of problem ($\mathcal{P}$);}\\
\mathcal{D}^* \; &: \; \text{optimal set of problem ($\mathcal{D}$).}\\
\end{align*}

\section{Certificate of Optimality}

Assume that we are given a Gramian decomposition of $p\in R_{2d}$ 
\begin{eqnarray*}
p = \sum_{i=1}^r \lambda_i (1+ v_{i,1}x_1+\cdots+v_{i,n}x_n)^{2d},
\end{eqnarray*}
corresponding to the points $\bvec{z}_i=(v_{i,1}, \ldots, v_{i, n})\in \R^n$ and $\lambda_i>0$ for $ i=1,\ldots, r$. Using the Vandermonde matrix $V_{d+1}$ of the points $\{{\bf z}_1,\ldots, {\bf z}_r\}$  and $\Lambda={\rm diag}(\lambda_1, \ldots, \lambda_r)$ as in  (\ref{Vandermonde}), it is clear that  $M_{d+1}=V_{d+1}^T \Lambda V_{d+1}$ is in the feasible set, $\mathcal{P}$. Our goal is to give conditions that guarantee that $M_{d+1}$ is in the set of optimal solutions, $\mathcal{P}^*$. To get such conditions we use both $(\mathcal{P})$ and $(\mathcal{D})$ defined above.

One can see that  $(\mathcal{D})$ is strictly feasible with $S=I$ and its optimum is bounded above by  ${\rm trace}(M_{d+1})$ since $M_{d+1}=V_{d+1}^T \Lambda V_{d+1}$ as a feasible solution for $(\mathcal{P})$. This implies that there is no duality gap between the optimal values of  $(\mathcal{P})$ and $(\mathcal{D})$, although $(\mathcal{D})$ might not attain its optimum \cite{vandenberghe1996semidefinite}. However, if we can construct a feasible pair $X\in\mathcal{P}$ and $(y, z, S)\in \mathcal{D}$  such that  $$<X,S>=0$$
then we must have $X\in \mathcal{P}^*$ and $(y, z, S)\in \mathcal{D}^*$ since  $$0=<X,S>= <I,X> - \mathbf{m}^T \mathbf{y},$$
which implies optimum by weak duality. Note that for positive semidefinite matrices $X$ and $S$ we have  $$<X, S> = 0 \iff XS=0.$$

Thus we get the following theorem: 

\begin{thm}\label{thm1}
The  moment matrix $M_{d+1} = V_{d+1}^T \Lambda V_{d+1}$  is optimal for $(\mathcal{P})$, or $M_{d+1} \in \mathcal{P}^*$, if there exists $ S \in {\mathcal S}_N$ such that:
\begin{align*}
&M_{d+1}S = 0,\\
&S= I-\sum_{ |\alpha| \leq 2d} y_{\alpha} Y_{\alpha} -\sum_{\substack{ |\alpha| \leq 2d+2, \\1 \leq i \leq \operatorname{dim}({\mathcal S}_\alpha)-1}} z_{\alpha,i} Z_{\alpha,i},\\
&S \succeq 0\\
\end{align*}
\end{thm}

Using Theorem \ref{thm1} we study when the optimal solution of $(\mathcal{P})$ is unique and $\mathcal{P}^* = \{M_{d+1}\}$. We are only concerned with cases where the rank $r$ symmetric decomposition of the associated polynomial $p$ is unique, and it is Gramian. In this case, Proposition \ref{unique} gives sufficient conditions to show $\mathcal{P}^* = \{M_{d+1}\}$: 

\begin{propo}\label{unique}
Assume that $p$ has Gramian rank $r$ and  the rank $r$ symmetric decomposition of $p$ is unique. If $\exists S$ satisfying Theorem \ref{thm1} of rank $N-r$, then $\mathcal{P}^* = \{M_{d+1}\}$.
\end{propo}
\begin{proof}
Suppose $p$ has a unique Gramian rank $r$ decomposition and let $S$ be a matrix satisfying Theorem \ref{thm1} of rank $N-r$. Let $M \in \mathcal{P}^*$. Since $MS=0$ and $\rank{ S}=N-r$, we have  $\rank M\leq r$. But by the stopping criteria in Theorem \ref{thm:stop}, $M$ defines a rank $\leq r$ symmetric decomposition for $p$, so the uniqueness of the symmetric decomposition implies that $M=M_{d+1}$.\end{proof}

Additionally we note the following  about the set of matrices satisfying Theorem \ref{thm1}.

\begin{propo}\label{prop:rank}
If $\exists S$ satisfying Theorem \ref{thm1}, then $\exists \bar{S}$ satisfying Theorem \ref{thm1} with $\rank{\bar{S}} \leq \dim R_{d+1}-\dim R_d = {{n+d}\choose{d+1}}$.
\end{propo}
\begin{proof}
Suppose $\exists S$ satisfying Theorem \ref{thm1}. By zeroing the Schur compliment of the submatrix indexed by degree $d+1$ monomials, we can produce $\bar{S}$ with $\rank{\bar{S}} \leq \dim R_{d+1}-\dim R_d$.
\end{proof}

To better aid our analysis of the problem we reformulate Theorem \ref{thm1} into a  problem involving polynomials. Theorem \ref{thm:sos} gives an alternative formulation of Theorem \ref{thm1} by noticing that the polynomial $\mathbf{x}^T S \mathbf{x}$ do not depend on the $z_{\alpha, i}$ variables and interpreting the problem as a sum of squares decomposition.

\begin{thm}\label{thm:sos}
The  moment matrix $M_{d+1} = V_{d+1}^T \Lambda V_{d+1} \in \mathcal{P}$ corresponding to the points $\bvec{z}_1, \hdots, \bvec{z}_r$ is optimal, or is in $\mathcal{P}^*$, if there exists $q \in R_{2d+2}$ and $q_{\alpha} \in R_{d+1}$ for all $\alpha\in \N^n$ with $\abs{\alpha} = d+1$ such that:
\begin{align*}
q & = \sum_{\abs{\alpha}=d+1} q_\alpha^2\\
q_{\alpha}(\bvec{z}_i) & = 0, \text{ for all } 1\leq i \leq r, \abs{\alpha}=d+1\\
\coeff(q,x^{\beta}) & = \delta_{2 \vert \beta} \text{ for } \abs{\beta} = 2d+1, 2d+2,
\end{align*}
where $$\delta_{2 \vert \beta} = \begin{cases}
1 \quad \text{if } \exists \gamma \in \N^n \text{ such that } 2 \gamma = \beta\\
0 \quad \text{otherwise.}\\
\end{cases}$$
\end{thm}
\begin{proof}
We prove the equivalence of the criteria of Theorem \ref{thm1} and Corollary \ref{thm:sos}. 
First we prove that the conditions of Theorem \ref{thm1} imply the condition of Corollary \ref{thm:sos}. Assume there exists $S$ such that $M_{d+1}S = 0$, $S= I-\sum y_{\alpha} Y_{\alpha} -\sum z_{\alpha,i} Z_{\alpha,i}$, and $S \succeq 0$ as in Theorem \ref{thm1}. Without loss of generality, from Proposition \ref{prop:rank} we assume $\rank{S} \leq {{n+d}\choose{d+1}}$ with Cholesky factorization $S=LL^T$. With $\mathbf{x} = [x^\beta]_{|\beta|\leq d+1}$, we let $q = \mathbf{x}^T S \mathbf{x}$ and let the collection $q_\alpha$ consist of the polynomials $L^T \bvec{x}$. Then $q = \bvec{x}^T S \bvec{x} = \bvec{x}^T L L^T \bvec{x} = \sum_{\alpha} q_{\alpha}^2 $, and each $q_\alpha$ vanishes on $\bvec{z_i}$ since $M_{d+1} S=0 \implies V_{d+1}^T \Lambda V_{d+1} L L^T = 0 \implies V_{d+1} L = 0$. Using the observations that
$$ <\mathbf{xx}^T, I> = \mathbf{x}^T \mathbf{x}, \quad <\mathbf{xx}^T, Y_\alpha> = x^\alpha, \quad<\mathbf{xx}^T, Z_{\alpha,i}> = 0,$$ we conclude that for $\abs{\beta}=2d+1,2d+2$, we have,
\begin{align*}
\operatorname{coeff}(q,x^{\beta}) & = \operatorname{coeff}(\bvec{x}^T ( I - \sum_{|\alpha|\leq 2d} y_{\alpha}Y_{\alpha}+\sum_{\substack{|\alpha| \leq 2d+2,\\ 1 \leq i \leq \operatorname{dim}({\mathcal S}_\alpha)-1}} z_{\alpha,i} Z_{\alpha,i} ) \bvec{x}, x^\beta)\\
& = \operatorname{coeff}(\bvec{x}^T I \bvec{x}, x^\beta)\\
& = \delta_{2\vert \beta}.
\end{align*}
Now we prove that the conditions of Theorem \ref{thm:sos} imply the conditions of Theorem \ref{thm1}. Assume there exists $q$ and $q_{\alpha}$ as in Theorem \ref{thm:sos}. Then we form a coefficient matrix, $L$, from the coefficient vectors of $q_\alpha$ and let $S = LL^T$ so $ S\succeq 0$. Also $q_{\alpha}(\bvec{z}_i)=0$ for $1 \leq i \leq r \implies V_{d+1} L = 0 \implies M_{d+1} S = 0$. To conclude, it is sufficient to show that the two sets
$$\left\{S\in {\mathcal S}_N\;:\; <Y_{\beta},S>=\delta_{2\mid\beta} \text{ for } |\beta|=2d+1,2d+2 \right\}$$ and $$\left\{S \in {\mathcal S}_N\; :\;S=I-\sum_{|\alpha|\leq 2d} y_{\alpha}Y_{\alpha}-\sum_{\substack{|\alpha| \leq 2d+2,\\ 1 \leq i \leq \operatorname{dim}({\mathcal S}_\alpha)-1}} z_{\alpha,i} Z_{\alpha,i}, \; y_\alpha, z_{\alpha,i} \in \R\right\}$$ are equal.
Above we proved the ``$\supseteq$" direction. Since both of these sets are affine spaces, it is enough to prove that the vector spaces $$\{S\in {\mathcal S}_N\;:\; <Y_{\beta},S>=0 \text{ for } |\beta|=2d+1,2d+2 \}$$ and $$\left\{S \in {\mathcal S}_N \;:\;S=\sum_{|\alpha|\leq 2d} y_{\alpha}Y_{\alpha}+\sum_{\substack{|\alpha| \leq 2d+2,\\ 1 \leq i \leq \operatorname{dim}({\mathcal S}_\alpha)-1}} z_{\alpha,i} Z_{\alpha,i}, \; y_\alpha, z_{\alpha,i} \in \R \right\}$$ have the same dimension.
By construction, we have that 
$\{ Y_{\alpha},Y_{\beta},Z_{\gamma,i}\;:\; |\alpha|\leq 2d, |\beta|=2d+1,2d+2, |\gamma| \leq 2d+2, 1 \leq i \leq \operatorname{dim}({\mathcal S}_\gamma) -1\}$ is a basis for ${\mathcal S}_N$, which proves the claim.
\end{proof}

Alternatively we consider the optimality of $M_{d+1}$ by utilizing a change of basis of the complimentary solution.
The next proposition connects the vanishing ideal of $r$ real points to the kernel of the Vandermonde matrix of the points. 
\begin{propo}\label{prop:K}
Let $V_{d}, V_{d+1}$ be the Vandermonde matrices of $r$ real points of degrees $d$ and $d+1$, respectively, and assume that 
$$
{\rm rank} (V_d)= {\rm rank }(V_{d+1})=r.
$$
Let $K_{d}$ be a matrix with columns that form a basis for $\operatorname{Ker}(V_d)$. Then there exists a  matrix  $F$ of size $(\dim R_d) \times (\dim R_{d+1}-\dim R_d)$ such that the columns of 
$$K_{d+1} := \begin{bmatrix} K_d & -F \\ 0 & I \end{bmatrix}$$
form a basis for $\operatorname{Ker}(V_{d+1})$.\\
Moreover, the vanishing ideal of our $r$ points is generated by polynomials corresponding to the columns of $K_{d+1}$ and  the normal forms of the monomials of degree $d+1$ modulo this vanishing ideals correspond to the columns of $F$.  
\end{propo}

\begin{proof} The first statement simply follows from the assumption that ${\rm rank} (V_d)= {\rm rank }(V_{d+1})$. To prove the second statement, we note that since $M_{d+1} = V_{d+1}^T \Lambda V_{d+1}$ with $\Lambda\geq 0$ we have $\operatorname{Ker}(V_{d+1}) = \operatorname{Ker}(M_{d+1})$ and  ${\rm rank} (M_d)= {\rm rank }(M_{d+1})=r$. By \cite{CurFial1996,LaLaRo2008}, the polynomials corresponding to the kernel of $M_{d+1}$ form a so called {\em border basis} for the vanishing ideal of the $r$ points, which implies the second statement.
\end{proof}

Theorem \ref{thm2} gives an alternative formulation of Theorem \ref{thm1}  by using the matrix $K_{d+1}$ as described in Proposition \ref{prop:K}. 


\begin{thm}\label{thm2}  Consider the  moment matrix $M_{d+1} = V_{d+1}^T \Lambda V_{d+1}$, and fix a matrix $K_{d+1}$ as described in Proposition \ref{prop:K}. Then the moment matrix $M_{d+1} \in \mathcal{P}^*$ if there exists a symmetric block matrix $G \in {\mathcal S}_{N-r}$ such that:
\begin{align*}
& G = \begin{bmatrix} * & g\\ g^T & I - \sum_{\substack{ |\alpha| = 2d+2, \\1 \leq i \leq \operatorname{dim}({\mathcal S}_\alpha)-1}} z_{\alpha,i} \tilde{Z}_{\alpha,i} \end{bmatrix}\\
&\operatorname{coeff}(\mathbf{x}^T K_{d+1} G K_{d+1}^T \mathbf{x}, x^\beta) = 0 \text{ for } |\beta| = 2d+1\\
&G \succeq 0.
\end{align*}
where $g$ is a real matrix of size ${{n+d}\choose{d}}-r$ by ${n+d}\choose{d+1}$, $I$ is the identity matrix of size ${n+d}\choose{d+1}$, $\tilde{Z}_{\alpha,i}$ is the submatrix of $Z_{\alpha,i}$ of degree $2d+2$ monomials, and $*$ is a real symmetric matrix of size ${{n+d}\choose{d}}-r$.
\end{thm}

\begin{proof}
We prove the equivalence of the criteria of Theorem \ref{thm2} and Theorem \ref{thm1}. Assume there exists $S$ such that $M_{d+1}S = 0$, $S= I-\sum y_{\alpha} Y_{\alpha} -\sum z_{\alpha,i} Z_{\alpha,i}$, and $S \succeq 0$ as in Theorem \ref{thm1}. Since $M_{d+1}S=0$ and the rows of $K_{d+1}$ form a basis for the left and right $\operatorname{Ker}(M_{d+1})$, we have $S = K_{d+1} G K_{d+1}^T$ for some $G \in {\mathcal S}_{N-r}$. With $\mathbf{x} = [x^\beta]_{|\beta|\leq d+1}$ as above, we have
$$ <\mathbf{xx}^T, I> = \mathbf{x}^T \mathbf{x}, \quad <\mathbf{xx}^T, Y_\alpha> = x^\alpha, \quad<\mathbf{xx}^T, Z_{\alpha,i}> = 0.
$$ Using $S= K_{d+1} G K_{d+1}^T = I-\sum y_{\alpha} Y_{\alpha} -\sum z_{\alpha,i} Z_{\alpha,i}$ we construct the polynomial $$\mathbf{x}^TK_{d+1}G K_{d+1}^T\mathbf{x}=\mathbf{x}^T\mathbf{x} - \displaystyle\sum_{|\alpha| \leq 2d} y_\alpha x^\alpha.$$
Therefore, $\operatorname{coeff}(\mathbf{x}^TK_{d+1}G K_{d+1}^T\mathbf{x},x^\beta) = 0$ for $|\beta|=2d+1$. Since $S \succeq 0$ and $K_{d+1}$ full row rank, we also have $G \succeq 0$.\\

Conversely, assume there exists $G \in {\mathcal S}_{N-r}$ with the form 
$$
G = \begin{bmatrix} * & g\\ g^T & I - \sum_{\substack{ |\alpha| = 2d+2, \\1 \leq i \leq \operatorname{dim}({\mathcal S}_\alpha)-1}} z_{\alpha,i} \tilde{Z}_{\alpha,i} \end{bmatrix}$$
 such that $\operatorname{coeff}(\mathbf{x}^T K_{d+1}^T G K_{d+1} \mathbf{x}, x^\beta) = 0 \text{ for } |\beta| = 2d+1$ and $G \succeq 0$ so that $S=K_{d+1}^TG K_{d+1} \succeq 0$. Additionally, using the identities above it is apparent that $\operatorname{coeff}(\mathbf{x}^T K_{d+1}^T G K_{d+1} \mathbf{x}, x^{2\beta}) = 1 \text{ for } |\beta| = d+1$. Finally, we use the same argument as in the proof of Theorem \ref{thm:sos} to conclude the proof. 
\end{proof}



\section[Sufficient Cond. for Optimality]{Sufficient Conditions for Optimality}\label{sect:opt}
In this section we demonstrate that in some special cases $M_{d+1}$ will generically be optimal in $(\mathcal{P})$ by imposing an assumption on the polynomials $q_\alpha\in R_{d+1} $ in Corollary \ref{thm:sos}, namely,  we assume that the degree $d+1$ part of $q_{\alpha}$ is equal to  $x^{\alpha}$.
\begin{cor}\label{cor:sosassumption}
The  moment matrix $M_{d+1} = V_{d+1}^T \Lambda V_{d+1} \in \mathcal{P}$ corresponding to the points $\bvec{z}_1, \hdots, \bvec{z}_r$ is optimal if there exists $q \in R_{2d+2}$ and for all $\alpha\in \N^n$ with $\abs{\alpha} = d+1$ there exist 
\begin{equation}\label{assump}
q_{\alpha} = x^{\alpha} + \text{lower degree terms}  \in R_{d+1}
\end{equation}
such that:
\begin{align*}
q & = \sum_{\abs{\alpha}=d+1} q_\alpha^2\\
q_{\alpha}(\bvec{z}_i) & = 0, \text{ for all } 1\leq i \leq r, \abs{\alpha}=d+1\\
\coeff(q,x^{\beta}) & = 0 \text{ for } \abs{\beta} = 2d+1.
\end{align*}
\end{cor}
\begin{proof}
Suppose there exists $q$ and $q_{\alpha} = x^{\alpha} + l.d.t.$ satisfying Corollary \ref{cor:sosassumption}, then $\operatorname{coeff}(q,x^\beta) = \delta_{2 \vert \beta}$ for $\abs{\beta} = 2d+2$ because degree $2d+2$ terms only depend on the squares of the degree $d+1$ terms in $q_\alpha$.
\end{proof}
Assumption (\ref{assump}) on $q_\alpha$ simplifies the criteria sufficient to prove optimality of $M_{d+1}$ into the solvability of a linear system. We note here that $\operatorname{Ker}(V_{d+1}) = \operatorname{Ker}(M_{d+1})$, so we can use the two interchangeably.

Using $K_{d+1}$ defined in Proposition \ref{prop:K}, we can look at a matrix existence formulation of Corollary \ref{cor:sosassumption} that is analogous to Theorem \ref{thm2}.

\begin{cor}\label{cor:assumption}
Let $K_{d+1}$ be as in Corollary \ref{cor:sosassumption}. The  moment matrix $M_{d+1} = V_{d+1}^T \Lambda V_{d+1} \in \mathcal{P}^*$ if there exists $G \in {\mathcal S}_{N-r}$ such that:
\begin{align*}
& G=\begin{bmatrix} gg^T & g\\ g^T & I\\ \end{bmatrix}\\
&\operatorname{coeff}(\mathbf{x}^T K_{d+1} G K_{d+1}^T \mathbf{x}, x^\beta) = 0 \text{ for } |\beta| = 2d+1.\\
\end{align*}
where $g$ is a real matrix of size $\dbinom{n+d}{n}-r$ by $\dbinom{n+d}{d+1}$ and $I$ is the identity matrix of size $\dbinom{n+d}{d+1}$.
\end{cor}
\begin{proof}
$G$ is clearly positive semidefinite with the decomposition $G = \begin{bmatrix} g\\ I\\ \end{bmatrix} \begin{bmatrix} g^T & I\\ \end{bmatrix}$. Using $G$ we let $q = \bvec{x}^T K_{d+1} G K_{d+1}^T \bvec{x}$ and associate each $q_{\alpha}$ with the corresponding element of the vector $\bvec{x}^T K_{d+1} \begin{bmatrix} g\\ I\\ \end{bmatrix}$. Then $q  = \sum_{\alpha} q_\alpha^2$ by construction. Since $K_{d+1}$ is in the null space of $V_{d+1}$ we also conclude that
$q_{\alpha}(\bvec{z}_i)  = 0, \text{ for all } 1\leq i \leq r, \abs{\alpha}=d+1$. By the structure of the matrices $K_{d+1}$ and $G$ we have that the degree $2d+2$ part of the polynomial $\mathbf{x}^T K_{d+1} G K_{d+1}^T \mathbf{x}$ is equal to $\sum_{|\alpha|=d+1} \left(x^\alpha\right)^2$. Lastly,
$\operatorname{coeff}(\mathbf{x}^T K_{d+1} G K_{d+1}^T \mathbf{x}, x^\beta) = 0 \text{ for } |\beta| = 2d+1 \implies \coeff(q,x^{\beta}) = 0 \text{ for } \abs{\beta} = 2d+1$.
\end{proof}

\begin{propo}\label{linear}
The values of $g$ satisfying Corollary \ref{cor:assumption} are the solution of an inhomogeneous linear system of equations.
\end{propo}

\begin{proof}
Let $q = \bvec{x}^T \begin{bmatrix} K_d g - F \\ I \end{bmatrix} \begin{bmatrix} g^TK_d^T - F^T & I \end{bmatrix} \bvec{x}$ and consider the degree $d+1$ polynomials in the row vector, $\mathbf{x}^T \begin{bmatrix} K_d g - F \\ I \end{bmatrix}$. The degree $d+1$ components of these polynomials consist of a single monomial that is independent of $g_{i,j}$. The degree $d$ coefficients of these polynomials are inhomogeneous but linear in $g_{i,j}$. Because degree $2d+1$ coefficients of $q$ rely only on the product of degree $d$ and degree $d+1$ coefficients of the polynomials the values of $g$ satisfy an inhomogeneous linear system.
\end{proof}
 
In order for  the inhomogeneous  linear system of Proposition \ref{linear} to have a solution, it is sufficient that the corresponding homogeneous equations are linearly independent. Thus we should try to understand what is the coefficient matrix of this linear system and determine when it is full row rank. For this we first define the notion of a subresultant matrix. Subresultant matrices of homogeneous polynomials $h_1, \ldots, h_t$ play an important role in studying the homogeneous parts of the ideal $\langle h_1, \ldots , h_t\rangle$.
.

 \begin{definition} Let $h_1, \ldots, h_t\in R_{d}$ be homogeneous polynomials for some $t\in \N$, and let $\Delta\geq d$. The degree $\Delta$ subresultant matrix of  $h_1, \ldots, h_t$  is the matrix whose columns are the coefficient vectors of the multiples of each $h_i$ with all monomials of degree $\Delta-d$. For example, if $\Delta-d = d+1$ and the monomials of degree $d+1$ are $\{x^{\alpha_i}\}_{i=1}^s$ for $s=\dim R_{d+1}-\dim R_d$, then
$$
{\rm Sres}_{\Delta}(h_1, \ldots, h_t):=\begin{array}{|c|c|c||c||c|c|c|}
\cline{1-7} 
& & & &&& \\
& & & &&& \\  
x^{\alpha_1}h_1& \dots &x^{\alpha_s}h_1&\cdots&x^{\alpha_1}h_t& \dots &x^{\alpha_s}h_t\\
& & && & & \\
& & & &&& \\ 
 \cline{1-7}
 \multicolumn{1}{c}{}
\end{array}.
$$
\end{definition}

\begin{thm} \label{resultant} Let $d, n, r$ be as above. Denote 
$$t:=\dbinom{n+d}{n}-r.
$$
Let $G = \begin{bmatrix} g\\ I\\ \end{bmatrix} \begin{bmatrix} g^T & I\\ \end{bmatrix}$ be a matrix satisfying Corollary \ref{cor:assumption} and $K_d$ be the matrix from Proposition \ref{prop:K}. Denote the entries of $g$ by $g_{i,j}$ for $i=1,\ldots, t$ and  $j=1, \ldots, \dbinom{n+d}{d+1}$, and denote the entries  of $K_d$  by $k_{i, \beta}$ for $i=1, \ldots, t$ and $|\beta|\leq d$. We define the homogeneous degree $d$ polynomials:
$$
h_i:=\sum_{|\beta|= d} k_{i, \beta} x^\beta \quad i=1, \ldots, t.
$$ Then the coefficient matrix of the linear system in Proposition  (\ref{linear}) in the variables $\{g_{i,j}\}$ is ${\rm Sres}_{2d+1}(h_1, \ldots, h_t)$.
\end{thm}

\begin{proof}
First note that the normal form coefficients only appear in the constant terms, so do not appear in the coefficient matrix. The rows of the coefficient matrix correspond to monomials $x^\beta$ of degree $|\beta| = 2d+1$. For each $j\in \{1, \ldots, {{n+d}\choose{d+1}}\}$ associate with it a unique monomial of degree $d+1$, $\alpha_j$. For fixed $i\in \{1, \ldots,t\}$ and $j\in \{1, \ldots, {{n+d}\choose{d+1}}\}$, the column corresponding to the variable $g_{i,j}$ has zero entry in the row corresponding to $x^\beta$ unless $x^{\alpha_j}$ divides $x^\beta$. If $x^{\alpha_j}|x^\beta$ then the entry is $k_{i, \beta-\alpha_j}$, which shows that the column of $g_{i,j}$ is the coefficient vector of $x^{\alpha_j}h_i$.  
\end{proof}

\begin{cor}\label{cor:sres} Let   $M_{d+1}=V_{d+1} \Lambda V_{d+1}^T$ as above,  with $V_{d+1}$ the Vandermonde matrix of $r$ real points, and we assume that $V_d^T$ has full column rank. Define the homogeneous degree $d$ polynomials $h_1, \ldots, h_t$ from ${\rm Ker}(V_d^T)$ as in Theorem \ref{resultant}. Then   $M_{d+1} \in \mathcal{P}^*$ if $Sres_{2d+1}(h_1, \ldots, h_t)$ has full row rank.
\end{cor}

In the rest of this subsection we  study when the rows of the subresultant matrix are independent. Note that the rows are independent if and only if 
\begin{eqnarray}\label{ideal}
\langle h_1, \ldots, h_t\rangle_{2d+1} = R_{=2d+1},
\end{eqnarray}
where the left hand side denotes the homogeneous part of degree $2d+1$ of the ideal generated by $h_1, \ldots, h_t$, and the right hand side denotes the space of homogeneous polynomials of degree $2d+1$. Thus (\ref{ideal}) is satisfied only if  $2d+1$ is greater or equal than the {\em regularity index} of $\langle h_1, \ldots, h_t\rangle$, i.e. the smallest degree where the Hilbert function of the ideal agrees with its Hilbert polynomial.  Note that if $h_1, \ldots, h_t$ has common roots in the projective space over $\C$ then (\ref{ideal}) can never be satisfied, which implies that we need to have $t\geq n$. 

For the rest of the section we assume that $h_1, \ldots, h_t$ is a system such that the dimension of  $\langle h_1, \ldots, h_t\rangle_{2d+1}$ is the maximum possible. In the results below we give specific constructions of particular real systems $h^*_1, \ldots, h^*_t$ and study when we have  $\langle h^*_1, \ldots, h^*_t\rangle_{2d+1} = R_{=2d+1}$. Therefore, if we assume that our $\langle h_1, \ldots, h_t\rangle_{2d+1}$ is maximal, then it will also imply that $\langle h^*_1, \ldots, h^*_t\rangle_{2d+1} = R_{=2d+1}$.

\begin{rema}
In \cite{Pardue2010} it was shown that systems $h_1, \ldots, h_t$ for which $\langle h_1, \ldots, h_t\rangle_{2d+1}$ is not maximal are defined by non-trivial polynomial equations, so over $\C$ they form a Zariski closed subset. However, even for the ``generic" case over $\C$, the behavior of $\langle h_1, \ldots, h_t\rangle_{2d+1}$ is not well understood.  In  \cite{Froberg1985} they give a conjecture about the Hilbert series of generic  systems over $\C$.
\end{rema}

The regularity index of $n\times n$ homogeneous systems were widely studied in the literature, but for highly overdetermined systems that has Hilbert series as in Fr\"oberg's conjecture in \cite{Froberg1985} only the  asymptotic behavior of the regularity index is known as $n\rightarrow \infty$ (c.f. \cite{BFS2004,BFS2013}).

The next theorem gives all values of $d$ and $n$  when (\ref{ideal}) is satisfied in the cases when $t= n$ and $t= n+1$. The analysis of the cases when $t>n+1$ is still ongoing. Since    $r=\dbinom{n+d}{n}-t$, we can easily translate these results in terms of the Gramian rank $r$.  Finally, we want to note that on the other end of the spectrum, when $t=\dbinom{n+d}{n-1}=\dim R_{=d}$ and $h_1, \ldots, h_t$ are generic, then the coefficient vectors of $h_1, \ldots, h_t$ form a square full rank matrix,  thus (\ref{ideal}) is satisfied for all $n$ and $d$. However in this case  $r=\dbinom{n+d-1}{n}=\dim R_{d-1}$, and the matrices $M_{d-1}$ and $M_d$ already satisfy the stopping criterion for flat extension, so we do not need an extension to $M_{d+1}$.

\begin{propo}\label{prop:fullrank} Let $h_1, \ldots, h_t$ be homogeneous polynomials of degree $d$ in $n$ variables, and assume that $\langle h_1, \ldots, h_t\rangle_{2d+1}$ is maximal. Then  
$
\langle h_1, \ldots, h_t\rangle_{2d+1} = R_{=2d+1}
$
if
\begin{enumerate}
\item  in the case of $t=n$
\begin{eqnarray*}
 &&n=2 \text{ for arbitrary } d,\\
&&n=3 \text{ and }d\leq 3, \\
&&n=4 \text{ and }d\leq 2,\\
&&n\geq 5  \text{ and } d=1. 
\end{eqnarray*}

\item in the case of $t=n+1$
\begin{eqnarray*}
 &&n=2 \text{ or } 3 \text{ for arbitrary } d,\\
&&n=4 \text{ and }d\leq 6, \\
&&n=5 \text{ and }d\leq 3,\\
&&n=6,7,8 \text{ and }d\leq 2,\\
&&n\geq 9  \text{ and } d=1.  
\end{eqnarray*}

\end{enumerate}

\end{propo}

\begin{proof} First note that if we find a particular system $h_1^*, \ldots, h^*_t$ of degree $d$ that satisfy $\langle h^*_1, \ldots, h^*_t\rangle_{2d+1} = R_{=2d+1}$, then any generic $h_1, \ldots, h_t$ will also satisfy it. For $t=n$, the standard theory of subresultants uses the system
$$
h_1^*:=x_1^d, \ldots, h_n^*:=x_n^d.
$$
Then one can define 
$$
\delta:= n(d-1),
$$
and it is easy to see that if $\Delta\geq \delta+1$ then the matrix ${\rm Sres}_{\Delta}(h^*_1, \ldots, h^*_n)$ has more columns than rows and contains the identity matrix, so it has full row rank. Thus we need that $2d+1\geq \delta+1$ and that is only satisfied in the cases listed in the claim. 

For $t=n+1$ we will use the system 
$$
h_1^*:=x_1^d, \ldots, h_n^*:=x_n^d, \;h_{n+1}^*:=(x_1+\ldots+x_n)^d.
$$
Let 
$$H_d(\nu):={\rm span}\{x^\gamma\;:\; |\gamma|=\nu, \;\forall i\; \gamma_i<d \}
$$
and denote by ${\mathcal{H}}_d(\nu):=\dim H_d(\nu)$. Clearly, the monomials not in $H_d(\nu)$ generate $\langle x_1^d, \ldots, x_n^d\rangle_\nu$. 
Define the linear map 
\begin{eqnarray*}
\psi_{h_{n+1}^*} : {H}_{d}(d+1)&\rightarrow&  {H}_{d}(2d+1)\\
x^{\alpha} &\mapsto&
x^{\alpha}\cdot h_{n+1}^* \;\;\;\;\mod \langle
x_1^{d}, \ldots, x_n^{d}\rangle_{2d+1}.
\end{eqnarray*}
 By \citep[Corollary 3.5 and Theorem 3.8.(0)][]{Wat}, the matrix of the  map $\psi_{h_{n+1}^*}$ has full rank. So if 
 \begin{eqnarray}\label{Hd}
 {\mathcal{H}}_d(2d+1)\leq {\mathcal{H}}_d(d+1)
 \end{eqnarray}
  then 
 $\psi_{h_{n+1}^*}$ is surjective, and ${\rm Sres}_{\Delta}(h^*_1, \ldots, h^*_n)$ has full row rank. Using the fact that
${\mathcal H}_{d}(\nu)={\mathcal H}_{d}(\delta-\nu)$ and that
${\mathcal H}_{d}(\nu)$ is monotonically decreasing in  $[\lceil\frac{\delta}{2}\rceil, \delta]$,  we get that (\ref{Hd}) is satisfied
when either $\lceil\frac{\delta}{2}\rceil\leq d+1\leq 2d+1$ or $\lceil\frac{\delta}{2}\rceil\leq \delta-(d+1)\leq 2d+1$. This is always satisfied if $n\leq 3$ and for $n\geq 4$ it reduces to $d\leq \frac{n+2}{n-3}$, resulting in the  values in the claim. 
\end{proof}

A different approach was presented in  \citep[Theorem 6][]{FroOttSha2012}, where they studied the minimal number $t$ such that a generic homogeneous form in $n$ variables of degree $kd$ is a sum of the $k$-th powers of $t$ forms of degree $d$ over $\C$. For the case of $k=2$ they prove that for $$t=2^{n-1}$$ and generic  $h_1, \ldots, h_t\in R_{=d}$ we have 
$$
\langle h_1, \ldots, h_t\rangle_{2d} = R_{=2d},
$$
which is slightly stronger than what we need in (\ref{ideal}). Moreover, their construction for $k=2$ works over the reals, in particular, they show that  the following $2^{n-1}$ real polynomials
$$
h_I^*:=\left(x_1+\sum_{i\in I} x_i- \sum_{j\not\in I} x_j\right)^d \quad \text{ for all } \; I\subseteq \{2, \ldots, n\}
$$
will generate $R_{=2d}$ in degree $2d$.  Moreover,  they show that there is an open subset of all real polynomials of degree $2d$  where the "typical rank"   is $2^{n-1}$, but there might be other "typical ranks" too (see also \cite{ComonOttaviani2009} on typical ranks over $\R$).   They also show that for large enough $d$ the  $t=2^{n-1}$ upper bound is sharp,  but for small $d$ this bound is not always sharp. 

\section[Cases When Md+1 is Never Optimal]{Cases When $M_{d+1}$ is Never Optimal}\label{sect:notopt}
In the previous section we explored cases of triplets $(n,d,r)$ where we can generically prove that $M_{d+1}$ is optimal for $(\mathcal{P})$ and list these cases. In this section we try to demonstrate that there are instances of values of $(n,d,r)$  where we expect that $M_{d+1}$ is not the optimum for  $\mathcal{(P)}$, or more precisely, we expect not to be able to find any $S$ satisfying Theorem \ref{thm1}. We demonstrate this  by simply counting the degrees of freedom and number of constraints in Theorem \ref{thm:sos}, without proving that these constrains are in fact linearly independent.

In Theorem \ref{thm:sos} there are a total of ${{n+d+1}\choose{d+1}}{{n+d}\choose{d+1}}$ coefficients, or degrees of freedom, coming from the ${{n+d}\choose{d+1}}$ polynomials. To eliminate symmetry of the system, we count $\frac{{{n+d}\choose{d+1}}({{n+d}\choose{d+1}}-1)}{2}$ constraints. Each of the polynomials is constrained to vanish on $r$ vanishing points, providing $r {{n+d}\choose{d+1}}$ constraints. An additional $ {{n+2d}\choose{2d+1}} + {{n+2d+1}\choose{2d+2}}$ constraints come from coefficient constraints on degree $2d+1$ and $2d+2$ monomials in the sum of squares.

This system is overconstrained if $$r {{n+d}\choose{d+1}} + {{n+2d}\choose{2d+1}} + {{n+2d+1}\choose{2d+2}} > {{n+d+1}\choose{d+1}}{{n+d}\choose{d+1}} - \frac{{{n+d}\choose{d+1}}({{n+d}\choose{d+1}}-1)}{2}.$$ Solving for $r$ we find that the system is overconstrained if $$r > \frac{{{n+d+1}\choose{d+1}}{{n+d}\choose{d+1}} - \frac{{{n+d}\choose{d+1}}({{n+d}\choose{d+1}}-1)}{2} - {{n+2d}\choose{2d+1}} - {{n+2d+1}\choose{2d+2}}}{{{n+d}\choose{d+1}}}.$$

Asymptotically, these bounds are not applicable due to the limitation that $r$ is less than the size of $M_d$, but there are instances where this bound is applicable.

One instance is when $n=2$ and $d=3$, the bound indicates that $M_{d+1}$ will generally not be optimal when $r=10$ as the linear system is overdetermined. This triplet of $(n,d,r)$ is a case where the corresponding decompositions are unique, and $\operatorname{rank}(M_d) = r$, but $M_{d+1}$ will generically not be optimal for $\mathcal{P}$.

\section{Uncertain Cases}

Outside of the cases listed in Sections \ref{sect:opt} and \ref{sect:notopt}, the possibility of $M_{d+1}$ being optimal in $P$ may depend on more than just the triplet $(n,d,r)$. Instances may depend fundamentally on the sets of points $\{\bvec{z}_i\}$. To demonstrate this we present two examples in the same triplet $(n,d,r)$ where one example has $M_{d+1}$ optimal, and one does not.

Let us consider the case when $n=2,d=3$. In this case, ${\rm size}(M_{d}) = 10$, and ${\rm size}(M_{d+1}) = 15$. A discussion of the extreme rays in this case can be found in \cite{blekherman2012nonnegative}. Gramian rank $10$ decompositions will generally not be optimal solutions in $(\mathcal{P})$ as the linear system in Corollary \ref{thm:sos} is overcomplete. Gramian rank $8$ decompositions will generically be optimal in $(\mathcal{P})$ from Corollary \ref{cor:sres} and Proposition \ref{prop:fullrank}. Between these two ranks we wish to understand what happens. Here we present two examples of Gramian rank $9$ decompositions, one where $\exists S$ satisfying Theorem \ref{thm1}, and one where $\nexists S$ satisfying Theorem \ref{thm1}.

\begin{example}
Let $\{\bvec{z}_i\} = \{(78, 87), (-45, 78), (-38, 32), (91, -76), (-18, 94), (-22, -22), (27, 99), (52, -16),\\ (-58, -87)\}$ be the set of $r=9$ points, and let $\lambda_i=1$ for $i=\{1, \hdots , 9 \}$. In this case $\exists S$ satisfying Theorem \ref{thm1} and $M_{d+1}$ is optimal in $(\mathcal{P})$. 
\end{example}

\begin{example}
Let $\{\bvec{z}_i\} = \{(-43, -34), (-18, -10), (-19, 23), (52, 72), (-66, -76), (48, -15), (35, 45),\\ (-83, -72), (51, 22)\}$ be the set of $r=9$ points, and let $\lambda_i=1$ for $i=\{1, \hdots , 9 \}$. In this case $\nexists S$ satisfying Theorem \ref{thm1} and $M_{d+1}$ is not optimal in $(\mathcal{P})$. In this instance, the optimal solution is rank 11.
\end{example}

These examples demonstrate the complexity of the cases where Proposition \ref{prop:fullrank} does not hold, as the solution to the relaxed problem may or may not be optimal in the original problem. In these cases the triplet $(n,d,r)$ is not sufficient to determine if $M_{d+1}$ is optimal in $(\mathcal{P})$ and specific information of the points is necessary.

\subsection{Future Work}

Some of the methods used to search for certificates of optimality also suggest future research avenues. Given an instance of a specific $(n,d,r)$ and $\{\bvec{z}_i\}$ the standard approach to search for a certificate of optimality using Theorem \ref{thm1} involves two steps. First, we solve the under-determined linear system $\operatorname{coeff}(\mathbf{x}^T S \mathbf{x}, x^\beta) = 0 \text{ for } |\beta| = 2d+1$. With the resulting affine solution, we look for an intersection with the positive semidefinite cone. 

Let us consider for a moment the set $T = \{ A ~ | ~ A \succeq 0, ~ A = I - \sum_{\substack{ |\alpha| = 2d+2, \\1 \leq i \leq \operatorname{dim}({\mathcal S}_\alpha)-1}} z_{\alpha,i} \tilde{Z}_{\alpha,i} \}$. We know that $I \in T$ with $z_{\alpha,i}=0$ for  $|\alpha| = 2d+2$ and $1 \leq i \leq \operatorname{dim}({\mathcal S}_\alpha)-1$, so the set is nonempty and has an interior. Another interesting observation is that this set is bounded.
\begin{thm}\label{bounded}
Fix $n$ and $d$ and an orthonormal basis $\{\tilde{Y}_{\alpha} \} \cup \{ \tilde{Z}_{\alpha,i} \}$ for $|\alpha| = 2d+2$ and $1 \leq i \leq \operatorname{dim}({\mathcal S}_\alpha)-1$. Then the set 
$$I - \sum_{\substack{ |\alpha| = 2d+2, \\1 \leq i \leq \operatorname{dim}({\mathcal S}_\alpha)-1}} z_{\alpha,i} \tilde{Z}_{\alpha,i} \succeq 0$$
is bounded.
\end{thm}
\begin{proof}
Choose $Y = \sum_{|\alpha| = 2d+2} y_{\alpha} \tilde{Y}_{\alpha} \succ 0$ to be a full rank positive definite matrix.  A generic collection of ${{n+d+1}\choose{d+1}}$ points will produce such a matrix. Suppose that there exists a matrix $Z = - \sum_{\substack{ |\alpha| = 2d+2, \\1 \leq i \leq \operatorname{dim}({\mathcal S}_\alpha)-1}} z_{\alpha,i} \tilde{Z}_{\alpha,i}$ such that $I+sZ \succeq 0$ for $s>0$, then $Z \succeq 0$. But $<Y,Z>=0$ by construction, therefore $Z=0$ and the set is bounded.
\end{proof}

Using these observations of $T$, we can make some conclusions from the solution of our linear system. For instance, if the solution to the linear system can be solved independent of the $z_{\alpha,i}$, then Corollary \ref{cor:assumption} applies and $M_{d+1}$ is optimal. Alternatively if the $z_{\alpha,i}$ variables are necessary, but yield a solution such that $\sum z_{\alpha,i}^2 < 1$, then $M_{d+1}$ will be optimal since $I - \sum_{\substack{ |\alpha| = 2d+2, \\1 \leq i \leq \operatorname{dim}({\mathcal S}_\alpha)-1}} z_{\alpha,i} \tilde{Z}_{\alpha,i} \succ 0$. Lastly, if the diameter of the set in Theorem \ref{bounded} is $\operatorname{diam}(n,d)$, then if the solution of the linear system closest to the origin has $\sum z_{\alpha,i}^2 > \operatorname{diam}(n,d)^2$ then Theorem \ref{thm1} cannot apply.

Studying the sets of $\{\bvec{z}_i\}$ that provide instances of each of these cases will be a topic for future exploration. Future research may also extend the idea further with a linear programming relaxation. The additional constraint $\sum \abs{ z_{\alpha,i}} < 1$ is linear and such a solution also guarantees the optimality of $M_{d+1}$ in $(\mathcal{P})$.

Additionally, we are interested in the rank of the optimal solutions in the cases when $M_{d+1}$ is not optimal in the minimal nuclear norm problem. One approach to address this question may be to examine the extremal rays in the feasible set of $\mathcal{P}$. This may provide a meaningful upper bound for the rank of the optimal solution, as $\mathcal{P}^*$ must contain an extremal ray.


\section{Conclusion}
In this paper we study the Gramian decomposition of tensors and polynomials by posing a rank optimization problem. Through relaxation of the optimization problem, we pose a convex optimization problem to approximate the minimal rank solution. Our analysis of the relaxed problem reveals a relation between the Gramian decomposition problem and the theory of subresultants. Our research further provides specific cases where the optimal solution to our relaxation is also minimum rank. Lastly we provide some interesting cases demonstrating the complexity of the problem and discuss future work.

\bibliographystyle{tfnlm}

\def\cprime{$'$} \def\cprime{$'$} \def\cprime{$'$}

\end{document}